\documentclass[12pt,reqno]{amsart}

\usepackage{graphicx}
\usepackage{amsmath}
\usepackage{amssymb}
\usepackage{amscd}
\usepackage{tikz}
\usepackage{pdfpages}
\usepackage{amsthm}
\usepackage{amsfonts}
\usepackage[utf8]{inputenc}
\usepackage[T1]{fontenc}
\usepackage{lmodern}
\usepackage[parfill]{parskip}
\usepackage{paralist}

\renewenvironment{itemize}[1]{\begin{compactitem}#1}{\end{compactitem}}
\renewenvironment{enumerate}[1]{\begin{compactenum}#1}{\end{compactenum}}

\calclayout

\newtheorem{theorem}{Theorem}
\newtheorem{theorema}{Theorem}
\newtheorem{theoremb}{Theorem}
\newtheorem{theoremc}{Theorem}

\newtheorem{theoreme}{Theorem}
\newtheorem{rk}[theorema]{Remark}
\newtheorem{lem}[theoremb]{Lemma}
\newtheorem{prop}[theoremc]{Proposition}

\newtheorem{dfn}[theoreme]{Definition}

\newcommand{\su}{\mathfrak{su}}
\newcommand\1{{\bf 1}}
\newcommand\ad{{\mathfrak{ad}}}
\newcommand\End{\op{End}}

\renewcommand\b{\beta}

\newcommand\C{{\mathbb C}}

\renewcommand\d{\delta}

\newcommand\g{\mathfrak{g}}
\newcommand\gl{\mathfrak{gl}}

\newcommand\h{\mathfrak{h}}
\renewcommand\H{\mathbb{H}}
\newcommand\Hh{\mathcal{H}}

\renewcommand\Im{{\textnormal{\texttt{Im}}}}

\newcommand\La{\Lambda}
\newcommand\m{\mathfrak{m}}
\renewcommand\k{\mathfrak{k}}

\renewcommand\O{{\mathbb O}}

\newcommand\op[1]{\mathop{\rm #1}\nolimits}
\newcommand\ot{\otimes}
\newcommand\p{\partial}

\renewcommand\p{\mathfrak{p}}
\newcommand\R{{\mathbb R}}

\renewcommand\Re{{\textnormal{\texttt{Re}}}}

\newcommand\so{\mathfrak{so}}
\renewcommand\sp{\mathfrak{sp}}
\newcommand\sym{\mathfrak{sym}}
\renewcommand\t{\tau}

\newcommand\T{\mathfrak{t}}

\newcommand\Z{{\mathbb Z}}

\newcommand\ii{\mathbf{i}}
\newcommand\jj{\mathbf{j}}
\newcommand\kk{\mathbf{k}}

\begin{document}

 \title[Submaximally Symmetric Quaternion Hermitian Structures]{Submaximally Symmetric\\ Quaternion Hermitian Structures}
 \author[B.\ Kruglikov \& H.\ Winther]{Boris Kruglikov$^\dag$, Henrik Winther$^\ddag$}
 \address{$\dag$ Institute of Mathematics and Statistics, UiT the Arctic University of Norway, Troms\o\ 90-37, Norway.
E-mail: \textsc{boris.kruglikov@uit.no}.}
\address{$\ddag$ Department of Mathematics and Statistics, Masaryk University, Kotl\'a\v{r}sk\'a 2, Brno  611 37, Czech Republic.
E-mail: \textsc{winther@math.muni.cz}.}
 \maketitle

 \begin{abstract}
We consider and resolve the gap problem for almost quaternion-Hermitian structures, i.e.\ we determine
the maximal and submaximal symmetry dimensions, both for Lie algebras and Lie groups, in the class of
almost quaternion-Hermitian manifolds. We classify all structures with such symmetry dimensions.
Geometric properties of the submaximally symmetric spaces are studied, in particular we identify
locally conformally quaternion-K\"ahler structures as well as quaternion-K\"ahler with torsion.
 \end{abstract}

\section{Introduction and main results}\label{S1}

An \emph{almost quaternionic structure} on a manifold $M$ is a smooth rank three subbundle $Q \subset \End(TM)$, which
locally possesses a basis $I,J,K$ with $I^2=J^2=K^2=IJK=-\1$.
An \emph{almost quaternion-Hermitian structure} on a manifold $M$ is an almost-quaternionic structure $Q$ together with
a Riemannian metric $g$ such that for any local almost complex structure $J\in\Gamma_\text{loc}(Q)$ the metric $g$ is $J$-Hermitian.

The class of almost quater\-nion-Hermitian structures contains, as a partial case, quater\-nion-Kähler structures
and hyper-Kähler structures \cite{B,Sal}, and there are other natural geometric classes \cite{C}.
This paper contributes to the study of the Lie group of automorphisms $\op{Aut}(M,g,Q)$ and the Lie algebra
of infinitesimal symmetries $\sym (M,g,Q)$ of these structures.

The quaternion K\"ahler spaces $\H P^n$, $\H^n$, $\H H^n$ (the middle is hyper-K\"ahler) admit the maximal symmetry
dimension among all almost quaternion-Hermitian structures of fixed quaternionic dimension $n=\frac14\dim M$:
this symmetry dimension equals
 $$
D_n=2n^2+5n+3
 $$
for both the group and the algebra of symmetries, see \cite{W} and Section \ref{S2} for details.
In \cite{Santa} large automorphism groups of almost quaternion-Hermitian manifolds were discussed but the
sharp upper bound for its submaximal dimension was not derived.

In this paper we resolve the problem of submaximal symmetry on both the algebra and the group level.
Let us note that the symmetry gap problem, to determine the difference between the maximal and submaximal symmetry dimensions,
has recently been in focus for many geometric structures, see e.g. \cite{KT,KWZ} and the references therein.
Usually the gaps for dimensions of group and algebra of symmetries are different. In our case they coincide.
We shall prove that the submaximal symmetry dimension is
 $$
d_n=\left\{\begin{array}{ll}2n^2+n+4, & n>2;\\ 2n^2+n+5, & n=1,2.\end{array}\right.
 $$
Our first result concerning the algebra of symmetries is as follows.
Note that for $n=1$ an almost quaternion-Hermitian structure is just a Riemannian metric, so the maximal
and submaximal symmetry dimensions are known:
$D_1=10$ is achieved on constant curvature spaces $S^4$, $\R^4$ or $H^4$ (with the standard metrics up to homothety),
while $d_1=8$ is achieved on the constant nonzero holomorphic curvature spaces \cite{E}
(i.e.\ $\C P^2$ with the Fubini-Study metric and $B^4\subset\H$ with the neutral pp-wave or their proportional metrics).
Henceforth we assume $n>1$.

 \begin{theorem}\label{ThA}
Let (M,g,Q) be a connected almost quaternion-Hermitian manifold. Assume that $\dim\sym(M,g,Q)<D_n$.
Then $\dim\sym(M,g,Q)\le d_n$. In the case of equality, if $n=2$ then $M$ is locally isomorphic to
the Wolf spaces $SU(4)/S(U(2)\times U(2))$ or $SU(2,2)/S(U(2)\times U(2))$.
If $n>2$ then the submaximally symmetric space $M$ is locally either one of the structures
$\Hh_1^\pm$, $\Hh_2$, $\Hh_3^\b$, $\Hh_4$, $\Hh_5^\b$,
admitting a simply transitive group of symmetries, which are classified in Section \ref{S32},
or one of two homogeneous spaces $Q_{\H P^{n-1}}$ and $Q_{\H H^{n-1}}$ modelled on the tautological
quaternionic bundles over $\H P^{n-1}$ and $\H H^{n-1}$, which are described in Section \ref{S33}.
The models are mutually non-equivalent.
 \end{theorem}

Left-invariant quaternion K\"ahler structures with negative scalar curvature on Lie groups were classified
by Alekseevskii \cite{A}. None of our submaximal models, except for dimension $n=2$, are quaternion K\"ahler.
We classify almost quaternion Hermitian spaces with submaximal symmetry in Section~\ref{S3}.
The invariant metrics on any of the models come in a two-dimensional family $g_{c_1,c_2}$ defined in formula \eqref{g}.

In Section \ref{S4} we investigate geometric properties of the models with submaximal symmetry. In particular, we discover
that all our submaximal models are quaternion K\"ahler with torsion, and moreover some of these models satisfy further integrability conditions, which yields new examples of locally conformally quaternion K\"ahler spaces, as well as examples
of spaces with intrinsic torsion supported in the irreducible $\sp(1)\sp(n)$-submodule $KH$ studied by
Salamon, Swann and Dotti-Fino (see \cite{Sal,C} and references therein).
In particular, all our submaximal symmetry models are quaternion Hermitian.

Note that for $n=1$ the quaternion K\"ahler condition is equivalent to the self-dual Einstein condition, and
for such structures the submaximal symmetry dimensions is again $d_1=8$.
For $n=2$ the submaximal bound $d_2=15$ is achieved on quaternion K\"ahler manifolds.
But for $n>2$ the symmetry gap for quaternion K\"ahler structures is bigger than the gap
$D_n-d_n=4n-1$ for the general almost quaternion Hermitian structures. Also, for
hyper-K\"ahler structures the gap remains unknown.

We remark that quaternion K\"ahler as well as hyper-K\"ahler structures are Einstein \cite{A,B},
and for Einstein positive definite metrics the gap is known: it is the same as for
Riemannian conformal structures \cite{KT}.
Our next result concerns the automorphism groups.

 \begin{theorem}\label{ThB}
Let (M,g,Q) be a connected almost quaternion-Hermitian manifold
different from $\H P^n$, $\H^n$, $\H H^n$ with their standard quaternion K\"ahler structure up to homothety.
Then $\dim\op{Aut}(M,g,Q)\le d_n$.
The submaximal symmetry dimension is achieved precisely as follows.
If $n=2$ then $M$ is the Wolf space $SU(4)/S(U(2)\times U(2))$ or $SU(2,2)/S(U(2)\times U(2))$.
If $n>2$ then $M$ is either one of the homogeneous spaces with submaximal symmetry from Theorem \ref{ThA},
or the locally flat quaternion K\"ahler space $\H^n/\Z=S^1\times\R^{4n-1}$, where $\Z\subset\R\subset\H\subset\H^n$,
or one of the quotients $\Hh_2/\Z$, $\Hh_5^0/\Z$, $Q_{\H P^{n-1}}/\Z$, $Q_{\H H^{n-1}}/\Z$ with its induced structure.
 \end{theorem}

This theorem is proved in Section \ref{S5}.

\textsc{Acknowledgement.} HW acknowledges hospitality and support of the Department of Mathematics and Statistics,
UiT the Arctic University of Norway. His work was supported by the grant P201/12/G028 of the Grant Agency of the Czech Republic.
The visit of HW to UiT was also supported by the project "Pure Mathematics in Norway" funded by Trond Mohn Foundation and  Tromsø Research Foundation.

\section{Maximal symmetry and the dimension gap}\label{S2}

We begin with local considerations, so let $\g$ be the Lie algebra of symmetries of $(M,g,Q)$.
Because of the invariant almost quaternionic structure, the isotropy algebra $\h\subset\g$ lies
in the parabolic subalgebra $\p_1\subset\gl(n+1,\H)$. Since $\h$ also preserves a Riemannian metric,
it is a subalgebra of the maximal compact subalgebra $\k=\sp(1)\oplus\sp(n)\subset\p_1$.

The isotropy representation of $\h$ on $T_xM$ is faithful and is equivalent to the restriction
$\h\hookrightarrow \k\subset\End(\H^n)$.
Thus the symmetry dimension of $(M,g,Q)$ is bounded by $D_n=\dim\k+\dim_\R\H^n=3+n(2n+1)+4n$.

\subsection{Maximally symmetric and Wolf spaces}\label{S21}

When the symmetry dimension is $D_n$, the isotropy should be equal to $\k$.
It follows from representation theoretic arguments \cite{Santa} that the holonomy group is contained in $\k$, i.e.\ $M$
is quaternion K\"ahler. By \cite{N} the holonomy algebra coincides with the isotropy algebra in the non-flat case,
and so $M$ is a symmetric space. Thus if $g$ has non-zero scalar curvature, $(M,g,Q)$ is locally isomorphic
to one of the quaternion K\"ahler symmetric spaces classified by Wolf \cite{W}. As we will need this classification
later, let us give the table here. We list only compact symmetric spaces, their non-compact duals have the
same dimensions and will not be required explicitly, see \cite{B} for details.
Below $\m=\g/\h$ is identified with $T_oM$, where $o$ is a marked point of the homogeneous space $M$.

	\begin{table}[h]
\centering
	\begin{tabular}{l | l | l | l}
$M=G/K$ & $\dim_\H\m$ & $\dim\g$ & Geom.\,interpretation\\
    \hline
$Sp(n+1)/Sp(1)Sp(n)$  & $n$  & $2n^2+5n+3$ & $\op{Gr}_1(n+1,\H)=\H P^n$ \\
$SU(n+2)/S(U(2)U(n))$ & $n$  & $n^2+4n+3$ & $\op{Gr}_2(n+2,\C)$ \\
$SO(n+4)/SO(4)SO(n)$  & $n\ge3$  & $\frac{(n+3)(n+4)}2$ & $\op{Gr}_4^+(n+4,\R)$ \\
$G_2/SO(4)$           & $2$  & $14$  & $\{\text{subalg. }\H\subset\O\}$ \\
$F_4/Sp(1)Sp(3)$      & $7$  & $52$  & $\{\text{subsp.\,}\H P^2\subset\O P^2\}$ \\
$E_6/SU(2)SU(6)$      & $10$ & $78$  & $\{\text{subsp.\,}(\C\ot\H)P^2\subset(\C\ot\O)P^2\}$\hspace{-1cm} \\
$E_7/Sp(1)Spin(12)$   & $16$ & $133$ & Rosenfeld plane $(\H\ot\O)P^2$ \\
$E_8/Sp(1)E_7$        & $28$ & $248$ & $\{\text{subsp.\,}(\H\ot\O)P^2\subset(\O\ot\O)P^2\}$\hspace{-1cm} \\
	\end{tabular}\hspace{1cm}\vspace{0.3cm}
\caption{Compact quaternion K\"ahler symmetric spaces.}\label{Wolf}
	\end{table}

Comparing dimensions and including the flat case we conclude that the maximal symmetry dimension $D_n$
is realized precisely on the Wolf spaces and the flat space with homogeneous representation
$Sp(1)Sp(n)\ltimes\H^n/Sp(1)Sp(n)$, as indicated in the introduction.
We can also deduce this from the reconstruction technique.

\subsection{Reconstructing homogeneous spaces}\label{S22}

If the symmetry algebra $\g$ acts locally transitively, then its Lie algebra structure can be algebraically recovered
from the isotropy representation $\h\subset\op{End}(\m)$, $\m=\g/\h$, and some other algebraic data \cite{KW}.
We summarize the reconstruction in the particular important case of a reductive Klein geometry.

Our isotropy $\h\subset\sp(1)\oplus\sp(n)$ is compact, hence reductive, and therefore a
complement $\m\subset\g$ exists that is $\h$-invariant, meaning $[\h,\m]\subset\m$.

The brackets $\La^2\h\to\h$, $\h\otimes\m\to\m$ are equivalent to the subalgebra structure and isotropy representation,
respectively. These encode the Jacobi identity with at least two arguments from $\h$.
The case of one argument from $\h$ and two from $\m$ is equivalent to the statement that the Lie bracket
$B:\La^2\m\to\g=\h\oplus\m$ is equivariant.
This bracket is still subject to the Jacobi identity with all three arguments from $\m$.

 \begin{dfn}
An \textit{invariant bracket} is an element $B$ of the space $(\Lambda^2 \m^\ast \otimes(\h \oplus \m))^\h$.
We split $B=B_\h+B_\m$ according to the values into the vertical and horizontal parts.
 \end{dfn}

If $B$ is vertical, i.e.\ $B_\m=0$, then $M=G/H$ is a locally symmetric space.
Let us demonstrate how to recover its structure in the maximal symmetry case.
This will give a local version of the result from Subsection \ref{S21} and also the brackets to be used later.

 \begin{prop}\label{maxdim}
If the symmetry dimension of $(M,g,Q)$ is $D_n$, then it is locally one of the quaternion K\"ahler spaces
$\H P^n$, $\H^n$, $\H H^n$.
 \end{prop}

 \begin{proof}
In the maximal symmetry case $\h=\k=\sp(1)\oplus\sp(n)$, $\m=\H^n$. So representing
$\m=R(\omega_1)\ot R(\pi_1)$ with $\omega_1$ referring to the fundamental weight of the $\sp(1)$ factor and $\pi_i$ referring to the fundamental weights of $\sp(n)$,
we get the decomposition of the $\k$-modules:
 \begin{gather*}
\La^2\m=R(2\omega_1)\oplus R(2\omega_1)\ot R(\pi_2)\oplus R(2\pi_1),\\
\m\oplus\h=R(2\omega_1)\oplus R(\omega_1)\ot R(\pi_1)\oplus R(2\pi_1).
 \end{gather*}
The modules $R(2\omega_1)$, $R(2\pi_1)$ are of real type, so by Schur's lemma the only $\k$-equivariant map
$\La^2\m\to\g$ annihilates the middle component and scales these two modules by real numbers.
The corresponding brackets $\Theta:\La^2\H^n\to\Im(\H)=\sp(1)$, $\Xi:\La^2\H^n\to\sp(n)$
via the metric $g(q_1,q_2)=\Re(q_1\bar{q}_2)$, $q_i\in\H^n$, have the form
 \begin{gather}
\Theta(q_1,q_2)=\sum_{a=\ii,\jj,\kk}g(q_1a,q_2)a, \label{Theta} \\
\Xi(q_1,q_2)q_3=\sum_{a=\1,\ii,\jj,\kk}g(q_1a,q_3)q_2a-g(q_2a,q_3)q_1a, \label{Xi}
 \end{gather}
The general bracket $B=c'\Theta+c\,\Xi$ satisfies the Jacobi identity ($\mathfrak{S}\equiv$ the cyclic sum)
 $$
\mathfrak{S}_{q_1,q_2,q_3}[[q_1,q_2],q_3]=(c'-2c)\mathfrak{S}_{q_1,q_2,q_3}\sum_{a=\ii,\jj,\kk}g(q_1a,q_2)q_3a=0
 $$
iff $c'=2c$. Re-scaling $\m$ transforms $c\mapsto\lambda^2c$, so we can reduce $B$ to $c(2\Theta+\Xi)$,
where $c=0,\pm1$. The first case is the flat structure, while the last two correspond to the Wolf spaces, as required.
 \end{proof}

We will also need another reconstruction, which contains the one above.

 \begin{prop}\label{inadmissible}
Suppose $\g$ acts locally transitively and $\h$ contains either a Cartan subalgebra of the first summand $\sp(1)$
or that of the ideal $\sp(n)$ in $\k=\sp(1)\oplus\sp(n)$. Then the space $(M,g,Q)$ is locally symmetric.
 \end{prop}

 \begin{proof}
The Cartan subalgebras of either ideal are all conjugate. Thus we need only to consider the cases $\T\subset \h$ where $\T$ is
a Cartan subalgebra of $\sp(1)$ or $\sp(n)$. With respect to $\sp(1)\oplus \sp(n)$, the module $\H^n$ upon complexification
is the tensor product
 $$
\H^n_\C = \C^2 \otimes_{\C} \C^{2n},
 $$
with $\sp(1)$ acting on the left factor and $\sp(n)$ on the right. Therefore $\sp(1)$, or its subalgebra $\T$, act on $\H^n$ with highest weight $\omega_1$. ($\H^n$ decomposes into a direct sum of $n$ equivalent modules with respect to either of them.) But then, in the module decomposition of $\La^2 \H^n$, we find only modules with highest weight either $2\,\omega_1$ or $0$, which means that due to Schur's lemma, there is no equivariant map $\Lambda^2\H^n \rightarrow \H^n$.

In the case $\T\subset\sp(n)$ there is a basis $t_1,\dots,t_n\in\T$ acting diagonally on $\C^{2n}=\bigoplus_{k=1}^n\C^2_k$,
i.e.\ $t_k$ acts with eigenvalues $\pm i$ on the subspace $\C^2_k$ and trivially on $\C^2_j$ for $j\neq k$.
Thus $\Lambda^2 \C^{2n}$ has a basis consisting of $\T$-eigenvectors, each of which either has eigenvalue $0$ for all $t_k$, or non-zero eigenvalues $\pm i$ with respect to precisely two Cartan elements. Neither type of eigenvector occurs in $\C^{2n}$.
This means that, again, there is no equivariant map $\Lambda^2\H^n\rightarrow \H^n$, and therefore the Lie bracket must map $\Lambda^2\m\rightarrow \h$, so $(\g,\h)$ is a symmetric pair. Then the homogeneous space must be locally symmetric.
 \end{proof}

\subsection{Symmetry dimension bound}\label{S23}

Propositions \ref{maxdim} and \ref{inadmissible} give sufficient conditions for an almost quaternion-Hermitian
structure $(g,Q)$ on a homogeneous space $M$ to be locally symmetric.
Aiming to investigate homogeneous structures that are not locally symmetric,
we call the subalgebras from Proposition \ref{inadmissible} \textit{inadmissible}.

 \begin{prop}\label{maximaladmissible}
The subalgebra $\h=\sp(1)_\text{\rm diag}\oplus\sp(n-1)\subset\sp(1)\oplus\sp(1)\oplus\sp(n-1)$ has the largest dimension
amongst the admissible subalgebras of $\k=\sp(1)\oplus\sp(n)$.
 \end{prop}

 \begin{proof}
By Proposition \ref{inadmissible} the projection of $\h$ to the ideal $\sp(n)\subset\k$ is injective and non-surjective.
Maximising its dimension is equivalent to maximising the dimension of this projection. Therefore we are looking for
maximal proper subalgebras of $\sp(n)$, and we will use Mostow's criterion \cite{M}, see also \cite[Chapter 6]{GOV}.
Subalgebras in $\sp(n)$ correspond to compact Lie algebras equipped with a faithful defining representation as
quaternionic-linear operators on $\H^n$, so $\h$ acts on $\H^n$.

Suppose the action is decomposable into non-trivial submodules. If one summand has quaternionic dimension 1,
then the algebra embeds into $\sp(1)\oplus\sp(n-1)$ that has dimension $\d_n=d_n-4n=2n^2-3n+4$.
Otherwise, each summand has dimension greater than 1 and the algebra embeds into $\sp(p)\oplus\sp(n-p)$ for some $1<p<n-1$.
Then $\dim\h\leq 2n^2-4np+4p^2+n <\d_n$.

Next, assume that $\h$ is simple and that it acts irreducibly on $\H^n$. The smallest irreducible quaternionic
representations can be deduced from \cite{O}, see the result in Table \ref{smalltable}.

	\begin{table}[h]
\centering
	\begin{tabular}{l | l | l | l}
Lie Algebra & Representation & $\dim_\R$ alg & $\dim_\H$ rep\\
	\hline
$\sp(n)$          & $\H^n$ & $n^2+n$ & $n$ \\
$\su(n), n\not=2$ & $\C^n \oplus \overline{\C^n}$ & $n^2-1$ & $n$ \\
$\so(n), n>5$     & $\R^n \otimes \H$ & $\tfrac{1}{2}(n^2-n)$ & $n$ \\
$\mathfrak{g}_2$  & $\R^7 \otimes \H $ & $14 $ & 7 \\
$\mathfrak{f}_4$  & $\R^{26} \otimes \H $ & $52$ & 26 \\
$\mathfrak{e}_6$  & $\C^{27}\oplus \overline{\C^{27}}$ & $78 $ & 27 \\
$\mathfrak{e}_7$  & $\H^{56} $ & $133 $ & 56 \\
$\mathfrak{e}_8$  & $\R^{248} \otimes \H $ & $248 $ & 248 \\
	\end{tabular}\vspace{0.3cm}
\caption{Smallest quaternionic representations of compact simple Lie algebras}\label{smalltable}
	\end{table}

For the classical families we have $\d_n >\dim\so(n)$ and $\d_n >\dim\su(n)$.
For the exceptionals we compare: $\d_7=81>\dim\mathfrak{g}_2$ and then
$\d_{248} >\d_{56} >\d_{27} >\d_{26} =1278 >\dim\mathfrak{e}_8  >\dim\mathfrak{e}_7 >\dim\mathfrak{e}_6
>\dim\mathfrak{f}_4$. This takes care of all simple irreducible subalgebras.

Next, consider semi-simple irreducible subalgebras. These are tensor products of irreducible representations
of the ideals in the Lie algebra, so $n$ is a composite number and the largest subalgebra is $\so(p)\oplus\sp(n/p)$,
where $p$ is the smallest factor of $n$. This has dimension $\frac{1}{2p^2}(p^4-p^3+4n^2+2np) < \d_n$, proving the claim.

Finally note that there is precisely one (up to $\k$-conjugation; for $n=2$ also $\h$-outer automorphism) admissible embedding
of $\h=\sp(1)\oplus\sp(n-1)$ into $\k$.
This embedding maps $\sp(1)$ onto the diagonal of the first two factors of $\sp(1)\oplus \sp(1)\oplus \sp(n-1)\subset \k$.
 \end{proof}

Now we get the dimension bound in general case, without assuming $M$ homogeneous.

 \begin{prop}
The submaximal symmetry dimension is bounded from above by $d_n$ and in the case of equality the space
$(M,g,Q)$ is locally homogeneous around generic points.
 \end{prop}

 \begin{proof}
If the isotropy is inadmissible, then $\g,\h$ is a symmetric pair and $(M,g)$ is a Riemannian symmetric space. The only
proper subalgebras in $\k$ of dimension $>\d_n$ are $\sp(n)$ and $\sp(1)\oplus\sp(1)\oplus\sp(n-1)$.
Then the quaternion Hermitian structure $(g,Q)$ on $\m$ is unique up to endomorphism,
and so it is a quaternion K\"ahler symmetric space. The same situation is
if dimension of the isotropy is $\d_n$ but the algebra is different from $\h$ of Proposition \ref{maximaladmissible}:
there are just two other embeddings of $\sp(1)\oplus\sp(n-1)$ into $\k$.

Table \ref{Wolf} implies that symmetry dimension of a Wolf space different from $\H P^n$ and $\H H^n$ is strictly smaller
than $d_n$ except for $n=2$, where $SU(4)$ and $SU(2,2)$ have dimensions $d_2$; also note that
$G_2$ and $G_2^*$ have dimensions $d_2-1$.

If the manifold is homogeneous, with admissible isotropy, the statement follows from Proposition \ref{maximaladmissible},
because $\dim\h\leq\d_n$ implies $\dim\g=\dim\h+\dim\m\leq d_n$.

If the symmetry does not act locally transitively, then $\m=\g/\h$ is the tangent space $T_o(G\cdot o)$ to the local orbit
at a regular point $o\in M$. Thus $\h\subset\k$ is a proper subalgebra with reducible defining representation and such that
the representation $\h\subset\op{End}(\m)$ is faithful. One easily checks that such $\h$ has dimension $\leq\d_n$.
Then $\dim\g<d_n$.
 \end{proof}

In the next sections we realize this dimensional bound, and hence prove that $d_n$ is actually the submaximal symmetry dimension.

\section{Construction of the Sub-maximal Models}\label{S3}

Let us first note that the case $n=2$ is special: by the results of Section \ref{S2} the almost quaternion Hermitian
space with submaximal symmetry dimension $d_2=15$ is locally symmetric and so one of the two models in Table \ref{Wolf}.
The quaternionic structures is standard and the metric is defined up to homothety.

The sub-submaximal symmetry dimension $d_2-1=14$ is realized either by symmetric spaces $G_2/SO(4)$, $G_2^*/SO(4)$
or by a homogeneous space with admissible isotropy. The latter follow the constructions for dimension $n>2$.
Henceforth in constructing the submaximal symmetry spaces we allow the general dimension $n\ge2$.

By the results of the previous section, the upper bound $d_n=\dim\h+4n$ on the symmetry dimension is attained only if the symmetry
algebra acts locally transitively, meaning that $\m=\g/\h$ is of dimension $4n$. We will construct $\h$-invariant structures $(g,Q)$ on $\m$,
which by the standard technique yields a homogeneous space $M$ with almost quaternion Hermitian $(g,Q)$ having
at least $d_n$ independent symmetries.

\subsection{The space of invariant brackets}\label{S31}

The restriction of the isotropy representation to $\h=\sp(1)\oplus\sp(n-1)\subset\k=\sp(1)\oplus\sp(n)$ branches $\m=\H^n$ into irreducibles. Indeed, fix a real orthonormal quaternion-compatible basis $1_1,i_1,j_1,k_1,\ldots,1_n,i_n,j_n,k_n$
of $\m$, so $\langle 1_p,i_p,j_p,k_p\rangle=\H\cdot 1_p$ for every $p$. Then
the isotropy algebra $\h$ is the annihilator of $1_1$:
 \begin{equation}\label{h-decomp}
\h\simeq \sp(1)\oplus \sp(n-1) = \text{Ann}(1_1)\subset \sp(1)\oplus\sp(n).
 \end{equation}
Recall that the $\sp(1)$ summand of $\h$ is neither the first summand nor a subalgebra of the second summand from
$\k$, these do not act trivially on $1_1$; it is the diagonal subalgebra with its adjoint action.
The module $\m$ decomposes into submodules with respect to $\h$,
 \begin{equation}\label{m-decomp}
\m = \R \oplus \Im(\H)\oplus \H^{n-1},
 \end{equation}
and in this decomposition the ideal $\sp(n-1)$ of $\h$ acts trivially on $\H=\R \oplus \Im(\H)$, while the ideal $\sp(1)$ of $\h$ acts trivially only on $\R=\langle 1_1\rangle$. We decompose with respect to $\h$:
 \begin{equation}\label{L2m-decomp}
\Lambda^2\m = \Im(\H)\oplus \H^{n-1} \oplus \Lambda^2\Im(\H) \oplus \Im(\H)\otimes \H^{n-1} \oplus\Lambda^2\H^{n-1},
 \end{equation}
where $\Lambda^2\Im(\H)=\ad_{\sp(1)}$,
$\Im(\H)\otimes\H^{n-1}$ decomposes further into $\H^{n-1}$ and the irreducible module
with the highest weight $3\omega_1+\pi_1$, while
$\Lambda^2\H^{n-1}$ decomposes further into $\ad_{\sp(1)}\oplus\ad_{\sp(n-1)}$ and the irreducible module
with the highest wight $2\omega_1+\pi_2$.

 \begin{prop}\label{invbrackets}
The space of invariant brackets has dimension $9$, among which there are
$\dim (\La^2\m^\ast \otimes \m)^\h=5$ horizontal and
$\dim (\La^2\m^\ast \otimes \h)^\h=4$ vertical brackets.
 \end{prop}

 \begin{proof}
The invariant horizontal brackets are $\h$-equivariant maps $\La^2\m\to\m$. The Lie algebra $\h$ is semi-simple,
so its modules are completely reducible. Some of the components are modules of quaternionic type over $\sp(n-1)$,
but all are of real type over $\h$. Thus by Schur's lemma there is one parameter in the bracket for each pair of
isomorphic irreducible components in $\La^2\m$ and $\m$. To count them complexify \eqref{m-decomp} :
 $$
\m\ot\C=\C\oplus S^2\C^2\oplus \C^2\ot\C^{2(n-1)}=R(0)\oplus R(2\omega_1)\oplus R(\omega_1)\ot R(\pi_1).
 $$
Similarly,
$(\Im(\H)\otimes\H^{n-1})\ot\C=R(2\omega_1)\ot R(\omega_1)\ot R(\pi_1)=(R(\omega_1)\oplus R(3\omega_1))\ot R(\pi_1)$
and $\Lambda^2\H^{n-1}\ot\C=\Lambda^2(R(\omega_1)\ot R(\pi_1))=R(2\omega_1)\oplus R(2\pi_1)\oplus R(2\omega_1)\ot R(\pi_2)$,
whence the complexification of \eqref{L2m-decomp} is
 $$
\La^2\m\ot\C= 3\cdot R(2\omega_1)\oplus 2\cdot R(\omega_1)\ot R(\pi_1)\oplus R(3\omega_1)\ot R(\pi_1)\oplus
R(2\pi_1)\oplus R(2\omega_1)\ot R(\pi_2).
 $$
Only the first (multiple) components contribute to the space of invariant horizontal brackets, and their dimension
is $3+2=5$.

Similarly, the invariant vertical brackets are $\h$-equivariant maps $\La^2\m\to\h$. From the complexification of \eqref{h-decomp}
 $$
\h\ot\C=S^2\C^2\oplus S^2\C^{2(n-1)}=R(2\omega_1)\oplus R(2\pi_1)
 $$
and the decomposition of $\La^2\m\ot\C$ we obtain $3+1=4$ independent invariant vertical brackets.
It is easy to check that all these brackets are real.
 \end{proof}

We give the formulae for the invariant brackets. The horizontal ones have the basis:
 \begin{itemize}
\item $\Theta: \H^{n-1}\ot\H^{n-1} \to \Im(\H)$, given by \eqref{Theta} with $n\mapsto(n-1)$,
\item $\Psi_1: \R\ot\Im(\H) \to \Im(\H)$, given by $\Psi_1(r,v)=r\,v$,
\item $\Psi_2: \R\ot\H^{n-1} \to \H^{n-1}$, given by $\Psi_2(r,q)=r\,q$,
\item $\Upsilon_1: \Im(\H)\ot\Im(\H) \to \Im(\H)$, given by $\Upsilon_1(v_1,v_2)=2\Im(v_1v_2)=v_1v_2-\bar{v}_2\bar{v}_1$,
\item $\Upsilon_2: \Im(\H)\ot\H^{n-1} \to \H^{n-1}$, given by $\Upsilon_2(v,q)=q\,\bar{v}$,
 \end{itemize}
where $r\in\R$, $v\in\Im(\H)$, $q\in\H^{n-1}$ (the same for indexed letters).

The first three vertical invariant brackets with the value in $\Im(\H)=\ad_{\sp(1)}$ have the same formulae
as $\Theta$, $\Psi_1$, $\Upsilon_1$ and the last bracket $\Xi:\H^{n-1}\ot\H^{n-1} \to \H^{n-1}$ is given by
\eqref{Xi} with $n\mapsto(n-1)$.

Note that there is one $\h$-invariant quaternionic structure $Q$ on $\m$ up to $\h$-endomorphisms. We will
fix it in what follows. There is also a Hermitian compatible invariant metric
$g_0 = g_\R + g_{\Im(\H)} + g_{\H^{n-1}}$ in terms of decomposition \eqref{m-decomp},
and a general $\h$-invariant almost quaternion Hermitian metric on $\m$
is given by 2 parameters $c_1,c_2\in\R_+$ as follows:
 \begin{equation}\label{g}
g = c_1 g_\R + c_1 g_{\Im(\H)} + c_2 g_{\H^{n-1}}.
 \end{equation}
Denote this metric by $g_{c_1,c_2}$.
The constants $c_1,c_2$ can be fixed by an endomorphism, but we keep this freedom to normalize the
structure constants of the Lie algebra $\g$ next.

\subsection{Submaximal structures on Lie groups}\label{S32}

Let us first note that the space $\m$ equipped with the bracket $[,]_\m=\Theta$ is a Lie algebra.
Indeed, it is two-step nilpotent, and all such brackets automatically satisfy the Jacobi identity.

Hence $\g=\h\ltimes\m$ equipped with the brackets $[,]_\h$, $[,]_\m$ along with $[h,m]=\varrho(h)m$,
$h\in\h$, $m\in\m$, where $\varrho$ is the isotropy representation of $\h$ on $\m$, is a Lie algebra.
Indeed, $\h$ is a sub-algebra of derivations of $\m$ and $\Theta$ is $\h$-invariant.

Thus the natural left-invariant Riemannian metric $g$ and almost quaternionic structure $Q$ on $M=\exp(\m)$
($\simeq\H^n$ as a topological space) has $\g$ as its symmetry algebra. One easily computes that the structure
is not locally flat or a Wolf space, which implies sharpness of the upper bound from Subsection \ref{S23}:
$\dim\sym(M,g,Q)=d_n$ for $n>2$ (for $n=2$ this gives sub-submaximal symmetry dimension $d_n-1$).

We shall first classify which horizontal invariant brackets from Proposition \ref{invbrackets}
give rise to Lie algebras in the same vein, thus constructing submaximal models via left-invariant structures
on Lie groups $M$ corresponding to the Lie algebras $\m$. The general invariant horizontal bracket on $\m$ is
 \begin{equation}\label{B-br}
B=\alpha\Theta + \beta_1\Psi_1+ \beta_2\Psi_2+ \gamma_1\Upsilon_1+ \gamma_2\Upsilon_2.
 \end{equation}

 \begin{prop}\label{4families}
The space $(\m,B)$ is Lie algebra with $\h$-invariant bracket precisely when the parameters in \eqref{B-br}
belong to one of the following four families:
 \begin{alignat*}{3}
&\mathcal{F}_1:& \quad \{\beta_1 = 2 \beta_2, \gamma_1 = \gamma_2 =0\} \qquad\qquad
&\mathcal{F}_2:& \quad \{\alpha=0, \gamma_1 = \gamma_2 = 0\} \\
&\mathcal{F}_3:& \quad \{\alpha=0, \beta_1 = 0, \gamma_1 = \gamma_2\} \qquad\qquad
&\mathcal{F}_4:& \quad \{\alpha=0, \beta_1 = 0, \gamma_2=0\}
 \end{alignat*}
 \end{prop}

 \begin{proof}
The bracket $B$ is skew symmetric and $\h$-invariant by construction, so we only consider the Jacobi identity
with all arguments from $\m$, which is equivalent to the system of six equations:
 \begin{alignat*}{3}
& \alpha (2\beta_2-\beta_1) = 0,\qquad && \alpha \gamma_1 = 0,\qquad && \beta_1 \gamma_1 = 0,\\
& \gamma_2(\gamma_1-\gamma_2) =0,\qquad && \alpha \gamma_2 = 0,\qquad && \beta_1 \gamma_2 = 0.
 \end{alignat*}
The solution set is the union of the above four families $\mathcal{F}_i$.
 \end{proof}

Note that the group of invertible $\h$-invariant endomorphisms of $\m$ is $\op{End}^\h(\m)=\R_\times^3$.
It is generated by a (nonzero) scaling in each component of decomposition \eqref{m-decomp}.
The quaternionic structure $Q$ on $\m$ is invariant iff the scaling factors of the components $\R\oplus\Im(\H)$
are equal. Hence we consider only admissible endomorphisms
$A_{s,t}:\R\oplus\Im(\H)\oplus\H^{n-1}\to\R\oplus\Im(\H)\oplus\H^{n-1}$,
$(r,v,q)\mapsto(sr,sv,tq)$, $s,t\in\R_\times=\R\setminus\{0\}$.

The (inverses of these) maps of $\m$ induce transformations of the parameters in \eqref{B-br}:
 $$
A_{s,t}(\alpha,\beta_1,\beta_2,\gamma_1,\gamma_2)=
(t^2s^{-1}\alpha,s\beta_1,s\beta_2,s\gamma_1,s\gamma_2).
 $$
Identifying the parameters under these transformations we get the following table\linebreak (Table~\ref{B-table})
of normalized parameters ($\beta\in\R$), where we exclude the flat case $M=\H^n$ (all parameters vanish).

	\begin{table}[h]
\centering
	\begin{tabular}{r|c|c}
\hspace{4pt} $\Hh_1^\pm$: \ $(\pm1,2,1,0,0)$  \hspace{8pt} &
\hspace{8pt} $\Hh_3^\b$ : \ $(0,2,\beta,0,0)$ \hspace{8pt} &
\hspace{8pt} $\Hh_5^\b$ : \ $(0,0,\beta,1,0)$ \hspace{4pt} \vphantom{$\dfrac{a}{a}$} \\
	\hline
\hspace{4pt} $\Hh_2$ : \hspace{7pt} $(1,0,0,0,0)$     \hspace{8pt} &
\hspace{8pt} $\Hh_4$ :    \ $(0,0,1,0,0)$     \hspace{8pt} &
\hspace{8pt} $\Hh_6^\b$ : \ $(0,0,\beta,1,1)$ \hspace{4pt} \vphantom{$\dfrac{a}{a}$}
	\end{tabular}\vspace{0.2cm}
\caption{Normalized $(\alpha,\beta_1,\beta_2,\gamma_1,\gamma_2)$ for the horizontal non-flat bracket $B$}\label{B-table}
	\end{table}

Note that all entries except for the last column yield a solvable Lie algebra $\m$,
while the last two with $\gamma_1=1$ correspond to an $\sp(1)$ Levi factor in $\m$.

 \begin{rk}\rm
The model $\Hh_6^\b$ is isomorphic to $\Hh_5^{\b}$.
Indeed, in both cases $\h$ and $\m$ contain a copy of subalgebra $\sp(1)$, which in the first case is the ideal.
Changing this ideal to the diagonal subalgebra $\sp(1)_\text{diag}\subset\sp(1)\oplus\sp(1)$ and
passing to a new $\h$-invariant complement $\R\oplus\sp(1)_\text{diag}\oplus\H^{n-1}$ we modify
the parameter $\gamma_2=1$ to $\gamma_2=0$; the metric parameters \eqref{g} change so:
$(c_1,c_2)\mapsto\bigl(\tfrac{\b-1}{\b+1}c_1,c_2\bigr)$. Henceforth we exclude the model $\Hh_6^\b$ .
 \end{rk}

Let $\g=\h\ltimes\m$ be the extension of $\m$ via derivations $\h\subset\mathfrak{der}(\m)$, where the last
embedding is via the isotropy representation $\varrho$, as described above. Let $M=\exp(\m)$ be the
simply-connected Lie group with Lie algebra $\text{Lie}(M)=\m$. We identify $T_eM=\m$.

 \begin{prop}\label{Pr7}
Every 5-tuple of parameters from Table \ref{B-table}, defining the bracket $B$ \eqref{B-br} on $\m$,
and the structures $Q$ and $g$ \eqref{g} define a left invariant almost quaternion-Hermitian structures $(g,Q)$ on $M$
with $\sym(g,Q)=\g=\h\ltimes\m$, with the only exception $\Hh_1^-$ and the metric defined by parameters $c_1=2c_2$,
in which case $\sym(g,Q)=\sp(1,n)$.
 \end{prop}

 \begin{proof}
For every set of parameters the existence of the model as well as its $\g$-invariance follows from the construction.
Let us explain why the symmetry algebra is precisely $\g$ in the non-exceptional cases.

Assume at first $n>2$. If the symmetry is larger, then $\g$ should homomorphically embed into the maximal symmetry algebra of dimension $D_n$.
The radical of $\g$ should then be embeddable into the maximal solvable subalgebra of $\sp(n+1)$, $\sp(1,n)$ or $\sp(1)\oplus\sp(n)\ltimes\H^n$,
that equals respectively to $\R^{n+1}$, $\mathfrak{rad}(\p)$ or $\R^{n+1}\ltimes\H^n$ (in which case the spectrum of the adjoint representation is purely imaginary). Here $\p$ is the (only up to conjugation) parabolic subalgebra of $\sp(1,n)$.
In the first and last cases $\mathfrak{rad}(\mathfrak{g})$ should embed into an Abelian algebra, which is impossible except for $\mathcal{H}_5^\beta$, $\beta=0$
(note that $\mathfrak{rad}(\mathcal{H}_5^\beta)$ does not embed into $\mathfrak{rad}(\mathfrak{p})$ for any $\beta$).
The radical of $\mathfrak{g}$ for $\mathcal{H}_i$, $1\leq i\leq 4$, coincides with $\mathfrak{m}$, which has the same dimension as $\mathfrak{rad}(\mathfrak{p})$ but different Lie algebra structure except for $\mathcal{H}_1^{-}$.

This latter case corresponds to the exceptional parameters $(-1,2,1,0,0)$ and the symmetry algebra is actually maximal, see Section \ref{S4}.
The other exception corresponds to $\beta=0$ in $\Hh_5^\beta$, in which case the radical of $\g$ is Abelian $\R\oplus\H^{n-1}$ and embeds into $\H^n$,
however in this case the semi-simple part $\sp(1)\oplus\sp(1)\oplus\sp(n-1)$ of $\g$ uniquely embeds into $\sp(1)\oplus\sp(n)\ltimes\H^n$
but the isotropy representations are different, so an embedding of the entire $\mathfrak{g}$ is not possible.

Finally, consider the special case $n=2$. The same arguments prove non-embedding of $\g$ into the symmetry algebra
of dimension $D_2=21$, but there are two algebras $\su(4)$ and $\su(2,2)$ of dimension $d_2+1=15$
that could contain $\g$. However these simple Lie algebras contain no subalgebras of codimension 1, and so
the claim is proved.
 \end{proof}

\subsection{Classification: the general brackets}\label{S33}

The computations above can be extended to include the 9-parametric bracket with mixture of horizontal and vertical parts,
but the formulae become messy. Instead we shall use the Levi decomposition of the resulting Lie algebra $\g$
with subalgebra $\h$ and quotient $\m$.

Let $\mathfrak{r}$ be the (solvable) radical of $\g$, and $\g_{ss}$ be the (semisimple) Levi factor.
The isotropy subalgebra $\h\subset\g$ is semi-simple. Therefore $\dim\mathfrak{r}\le4n$.

 \begin{prop}\label{P8}
If $\dim\mathfrak{r}=4n$ then $\g$ is equivalent to one of the algebras from Proposition \ref{4families} with
the bracket $B$ \eqref{B-br} given by Table \ref{B-table} with $\gamma_1=0$.
 \end{prop}

 \begin{proof}
The isotropy algebra $\h$ is semi-simple, hence the radical is complementary to it. The radical is also a $\g$-ideal,
so in particular an $\h$-submodule. This submodule must be equivalent to $\m$ and now Proposition \ref{4families}
implies the claim.
 \end{proof}

To complement the computations of the last section, we assume now that $\m$ is not a subalgebra in $\g$,
so that $\dim\mathfrak{r}<4n$. Thus with the bracket $B=B_\h+B_\m$ on $\m$, $B_\h\neq0$, a choice of
Levi factor $\g_{ss}$ should contain some components of decomposition \eqref{m-decomp} in addition to $\h$.
Indeed, $\g_{ss}$ is a sub-module over itself, hence also an $\h$-module, and so we can evoke the $\h$-decomposition
into irreducible pieces.

 \begin{prop}
The semi-simple Levi factor $\g_{ss}$ of $\g$ does not include the sub-module $\R$, and it has rank at most $n+1$.
 \end{prop}

 \begin{proof}
Since $\h$ is compact, its Cartan subalgebra can be embedded into a subalgebra of a Cartan subalgebra of $\g_{ss}$.
The trivial submodule in $\m$ of the Cartan subalgebra of $\h$ has dimension 2 and it is $\R\oplus\t$,
where $\t\subset\Im(\H)$. By Proposition \ref{invbrackets} there is no invariant bracket which takes values in $\R$,
but $\g_{ss}$ is a perfect Lie algebra. Hence $\R$ cannot be included.
The algebra $\h$ has rank $n$, hence $\g_{ss}$ has rank at most $n+1$.
 \end{proof}

 \begin{prop}\label{levifactor}
Suppose $\h$ is a proper subalgebra of the Levi factor $\g_{ss}$. Then $\g_{ss}$ is one of the following Lie algebras:
    \begin{enumerate}
\item $\sp(1)\oplus\sp(1)\oplus\sp(n-1)$,
\item $\sp(n)$ \ or\/ \ $\sp(1,n-1)$,
\item $\sp(1)\oplus\sp(n)$ \ or\/ \ $\sp(1)\oplus\sp(1,n-1)$.
	\end{enumerate}
 \end{prop}

 \begin{proof}
As before, the submodule $\R$ cannot be included in the Levi-factor. Thus $\g_{ss}$, as an $\h$-module, must be constructed from $\h$
and the submodules $\Im(\H), \H^{n-1}$. We consider the rank and dimension of each combination, and when a match is found we consider the embeddability of $\h$ as the final condition.

(1)
Module decomposition: $\g_{ss}=\h\oplus\Im(\H)$. The rank is $n+1$, and the dimension comparison implies that we have
$\g_{ss}=\h\oplus\sp(1)=\sp(1)\oplus\sp(1)\oplus\sp(n-1)$ or $\g_{ss}=\sp(2,\C)\oplus\sp(n-1)$.
The latter case is however impossible, because $\h$ embeds into such $\g_{ss}$ uniquely, $\sp(2,\C)$ acts trivially
on $\mathfrak{r}=\R\oplus\H^{n-1}$ and so the corresponding module $\g_{ss}\ltimes\mathfrak{r}/\h$ has $\H^{n-1}$
as the trivial $\sp(1)$-submodule contrary to decomposition \eqref{m-decomp}.

(2)
Module decomposition: $\g_{ss}=\h\oplus\H^{n-1}$. The rank is $n$. Consider at first the case when $\g_{ss}$ is simple.
Combined with the dimension, the candidate here is either a real form of $B_n$ or $C_n$,
or of $E_6$ (having the same dimension as $C_6$) for $n=6$.
The fact that $\h\subset \g_{ss}$ eliminates the possibility $\g=\so(p,q)$ for $p+q=2n+1$.
Indeed, in the opposite case
$\sp(2,\C)\oplus\sp(2n-2,\C)$ is represented on $\C^{2n+1}$ that is equal to either $S^2\C^2\oplus\C^{2n-2}$ for $n\ge2$
or $\C^2\oplus\C^{2n-1}$ for $n=2,3$ (with the unique choice of irreducible modules in components),
however those possess no invariant non-degenerate symmetric bilinear form.
Also, $\sp(1)\oplus\sp(5)$ does not embed into $\mathfrak{e}_6$ in view of Dynkin's classification of maximal subalgebras \cite[Chapter 6]{GOV}.

Next, if $\g_{ss}$ is the direct sum of simple Lie algebras of ranks 1 and $(n-1)$, then the dimension comparison implies that
the simple ideal of rank $(n-1)$ is of type either $F_4$ or $E_7$. Neither is realizable, as $\mathfrak{f}_4$ and $\mathfrak{e}_7$ do
not contain $\sp(4)$ and $\sp(7)$, respectively, as maximal subalgebras.
Thus we conclude that either $\g_{ss}=\sp(n)$ or $\g_{ss}=\sp(1,n-1)$.

(3)
Module decomposition: $\g_{ss}=\h\oplus\Im(\H)\oplus\H^{n-1}$.
The rank is $n+1$, and dimension is 3 higher than the previous case.
Consider at first the case when $\g_{ss}$ is simple. By dimensional comparison this is possible iff $n=3$ and
the algebra is of type $A_4$. However $\su(5)$, or any other real form of $A_4$, does not contain $\sp(1)\oplus\sp(2)$.

Next, assume $\g_{ss}$ is the sum of rank $2$ semi-simple and rank $(n-1)$ simple ideals. If the former has dimension 6
(can be either simple or semi-simple), then $\sp(4)$ or $\sp(7)$ will embed to, respectively, $\mathfrak{f}_4$ and $\mathfrak{e}_7$,
which was already ruled out. Otherwise the rank $2$ ideal is simple and the dimension comparison yields that $n=3$ with the
ideals being $B_2$ and $C_2$. This leads to $\g_{ss}=\g_2\oplus\sp(2)$, which might have been possible with the decomposition into
submodules $\g_2=\sp(1)\oplus\Im(\H)\oplus\H^2$ (note the quotient $\g_2/\sp(1)\oplus\sp(1)=\H^2$),
but the summand $\H^2\subset\g_{ss}$ is a non-trivial $\sp(2)$-module, and so this case is also ruled out.

Finally, if $\g_{ss}$ is the sum of rank $1$ and rank $n$ simple ideals, then by dimension reasons the first ideal
is $A_1$ and the second is either $C_n$ or $E_6$ in the particular case $n=6$.
The latter possibility is not realizable because there is no embedding $\sp(5)$ into $\mathfrak{e}_6$.
Therefore we conclude that $\g_{ss}=\sp(1)\oplus\sp(n)$ or $\g_{ss}=\sp(1)\oplus\sp(1,n-1)$ as claimed.
 \end{proof}

Let us introduce a quaternionic analogue of the twistor construction on a
quaternion K\"ahler space $(N,g,Q)$ of quaternionic dimension $(n-1)$. Consider the bundle $(\R\times Q)_\times$
over $N$ with the fiber $\H_\times\simeq\R_+Sp(1)$ and the total space $Q_N$. Since the Levi-Civita connection
of $g$ preserves the quaternionic structure $Q\subset\op{End}(TN)$, it induces a connection on this bundle
and hence the splitting $TQ_N=HQ_N\oplus VQ_N$ into horizontal and vertical parts. Both components are equipped
with quaternion Hermitian structures, and this induces an almost quaternion Hermitian structure on $Q_N$;
the metric is parametrized by real re-scalings of $VQ_N$. Note that this structure is not quaternion K\"ahler
even in the case when $N$ is a maximally symmetric quaternion K\"ahler space with nonzero scalar curvature.
The quaternionic dimension of $Q_N$ is $n$.

 \begin{theorem}\label{vertsolutions}\label{clas}
Let $\h=\sp(1)\oplus\sp(n-1)$ be represented on $\m=\g/\h$ as in \eqref{m-decomp}. Then
 \begin{itemize}
\item
Either $\g=\h\ltimes\m$ and $\m$ is a Lie algebra as described in Proposition \ref{4families}, or
\item
$\g=\H\oplus\sp(n)$ or $\H\oplus\sp(1,n-1)$, with $\H$ having the quaternion commutator.
 \end{itemize}
In the latter case $\sp(1)\subset\h$ is embedded in $\g$ diagonally, while $\sp(n-1)\subset\h$ is the standard embedding,
and $M$ is one of the spaces $Q_{\H P^{n-1}}$ or $Q_{\H H^{n-1}}$.
 \end{theorem}

 \begin{proof}
By Proposition \ref{P8} it suffices to consider the pairs $(\g,\h)$, when the radical of $\g$ has dimension $<4n$.
We use the numeration of these cases from Proposition \ref{levifactor}.

In case (1), $\g_{ss}=\sp(1)\oplus\sp(1)\oplus\sp(n-1)$ and $\sp(1)\subset\h$ is embedded diagonally into the
$\g_{ss}$-ideal $\sp(1)\oplus \sp(1)$. Taking one of the ideals $\sp(1)\subset\g_{ss}$ to be the submodule
$\Im(\H)\subset\m$ we conclude that $\m=\sp(1)\oplus\mathfrak{r}$ is an $\h$-invariant complement and a Lie algebra,
so it is among the cases of Proposition \ref{4families} with $\gamma_1\neq0$.

In case (2), $\g_{ss}=\sp(n)$ or $\sp(1,n-1)$ and the module $\Im(\H)$ is trivial, which does not meet the requirement
from decomposition \eqref{m-decomp}. Thus this case is ruled out.

In case (3), $\g_{ss}=\sp(1)\oplus\sp(n)$ or $\sp(1)\oplus\sp(1,n-1)$ and the radical $\mathfrak{r}=\R$, hence $\g$ must
be the reductive algebra $\g_{ss}\oplus \R$. The only freedom is how to inject $\h$ up to conjugation, and the module
structure \eqref{m-decomp} of $\g/\h$ tells that $\sp(1)\subset\h$ is diagonally embedded between the ideals
$\sp(1)$ and $\sp(n)$ of $\g$. Since $\H=\R\oplus\sp(1)$ as a Lie algebra, we are done.

The two models, corresponding to the second (properly homogeneous) possibility
are associated to quaternionic line bundles. Indeed, since the Lie algebra $\H=\R\oplus\sp(1)$ exponentiates to
$\R_+Sp(1)$ (as the simply connected model) the first of them is a bundle over the quaternionic projective space
 $$
M=\R_+Sp(1)Sp(n)/Sp(1)Sp(n-1)\to Sp(n)/Sp(1)Sp(n-1)=\H P^{n-1}
 $$
with the fiber $\R_+Sp(1)\simeq\H_\times=\H\setminus0$. This
$\H_\times$-fiber bundle is associated to the tautological $\H$-line bundle over $\H P^{n-1}$.
Similarly, for $\g=\R\oplus\sp(1)\oplus\sp(1,n-1)$ the corresponding homogeneous model $M=G/H$ is an $\H_\times$-fiber
bundle over $\H H^{n-1}$. This identifies the two models with $Q_{\H P^{n-1}}$ or $Q_{\H H^{n-1}}$.
 \end{proof}

We assert that the last two models have submaximal symmetry dimension as follows. By construction they possess
a symmetry algebra (and group) of dimension $d_n$. If the full symmetry is larger, then the given symmetry
algebra $\g=\R\oplus\sp(1)\oplus\sp(n)$ or $\R\oplus\sp(1)\oplus\sp(1,n-1)$
embeds into the maximal symmetry algebra, i.e.\ $\sp(n+1)$, $\sp(1,n)$ or $\sp(1)\oplus\sp(n)\ltimes\H^n$.
The latter case is impossible by the comparison of ranks of the Levi factors, while in the first cases the embedding
exists and is unique. However the $\h$-module structure of $\g/\h$ differs from \eqref{m-decomp}, implying the claim.

This remark and Theorem \ref{clas} finish the proof of Theorem \ref{ThA}.

 \begin{rk}\rm\label{RK2}
Choose an element $I\in\sp(1)\subset\h$, and define the twisted bracket $\Theta'=I\,\Theta\, I^{-1}$ on $\m$.
The nilpotent Lie algebra $(\m,\Theta')$ is equivariant only with respect to the centralizer
$Z_\h(I)=\so(2)\oplus\sp(n-1)\subset\h$. Thus the symmetry algebra $\g$ of the corresponding
almost quaternion-Hermitian structure $(g,Q)$ on $M=\exp(\m)$ is
$\sym(M,g,Q)= Z_\h(I)\ltimes \m$ of $\dim \g = 2n^2+n+2=d_n-2$ for $n>2$.
This raises the question whether $d_n-2$ is the sub-submaximal symmetry dimension,
or there exists a quaternion Hermitian manifold with the symmetry dimension $d_n-1$. 
 \end{rk}

\section{Geometry of the Sub-maximal Models}\label{S4}

In this section we investigate geometric properties of the models obtained in Section \ref{S3}.
For $n=2$ the submaximally symmetric spaces are Wolf spaces, and so the structure $(g,Q)$ is quaternion K\"ahler,
while the metric $g$ is Einstein with parallel curvature.
Henceforth for the rest of this section we assume $n>2$.

\subsection{First order integrability conditions}

We will now consider some integrability conditions for almost quaternion-Hermitian structures in the context of our sub-maximal models. In particular, we are interested in the existence of examples which are:
\begin{itemize}
	\item Quaternion Kähler
	\item Locally conformally Quaternion Kähler
	\item Quaternion Kähler with torsion
\end{itemize}
The class quaternion K\"ahler with torsion was introduced in \cite{IM}, and consists of quaternion Hermitian structures admitting a quaternionic metric connection with totally skew-symmetric torsion of type $(2,1)+(1,2)$ with respect to any local almost complex structure from $Q$.
Let $I,J,K$ be a an adapted local frame for the quaternionic structure $Q$, and let $\omega_A$ for $A=I,J,K$ be the associated two-forms given by $\omega_A(X,Y)=g(X,AY)$.
Natural differential equations for these conditions, and more, are given in \cite{C}, in terms of the
fundamental four-form $\Omega$ of the structure, where $\Omega$ is given by
 \begin{equation*}
\Omega = \omega_I \wedge \omega_I+ \omega_J \wedge \omega_J+\omega_K \wedge \omega_K.
 \end{equation*}
Note that while $\omega_A$ depends on an arbitrary choice and is not $\sp(1)$-invariant, $\Omega$ is invariantly defined.
This is the main tensorial invariant for almost quaternion-Hermitian structures.
The particular conditions we are interested in are given in Table \ref{diffcons}, in which the form $\xi$
is uniquely defined by the given conditions, but can be also explicitly given together with 1-forms
$\xi_I,\xi_J,\xi_K$ through the codifferential $\delta$ as follows:
 $$
\xi= -\tfrac{1}{6(2n+1)}\sum_{A=I,J,K}\langle A^* \delta\Omega,\omega_A\rangle_g,\quad
\xi_A=-\tfrac{3}{2(n-1)}\xi-\tfrac1{4(n-1)}\langle A^*\delta\Omega,\omega_A\rangle_g.
 $$

 \begin{table}[h]
	\centering
 \begin{tabular}{l | l | l}
Class of structure  & $EH$-formalism & differential equation\\
		\hline
Quaternion K\"ahler (QK) & $0$  & $d\Omega=0$ \vphantom{$\frac{A^A}A_A$} \\
Locally Conformally QK & $EH$ & $d\Omega = \xi \wedge \Omega$ for some $\xi \in \Omega^1(M)$\\
QK with special torsion & $KH$ & $d\Omega =\frac{1}{3}\sum_{A=I,J,K} i_A(d^\ast \Omega) \wedge \omega_A$\\
&& \hskip27pt and $\xi_I=\xi_J=\xi_K$  \\
QK with torsion & $(K+E)H$ & $d\Omega =\frac{1}{3}\sum_{A=I,J,K} i_A(d^\ast \Omega) \wedge \omega_A -\xi \wedge \Omega$\\
&& \hskip27pt for some $\xi \in \Omega^1(M)$
	\end{tabular}
	\vskip9pt \caption{Natural differential equations for some first order classes}
	\label{diffcons}
\end{table}

These conditions all happen to be \textit{first order classes}, and can equivalently be given by the vanishing of some subset of invariant projections for $\alpha\in S$
 $$
\pi_\alpha(d\Omega)=0,
 $$
where
 $
\Lambda^5 \m = \sum_{\alpha\in S} V_\alpha
 $
is the decomposition of the space of 5-forms on $M$ into simple modules with respect to the structure group $Sp(1)Sp(n)$.

Note that not every submodule in this decomposition corresponds to the intrinsic torsion (the torsion of a minimal
adapted connection). The latter is the invariant component of $d\Omega$ contained in $(E+K+\Lambda^3_0E )(H+S^3H)$,
or equivalently in the span of submodules
 \begin{equation}\label{6cl}
EH,\ KH,\ \Lambda^3_0E H,\ ES^3H,\ KS^3H,\ \Lambda^3_0ES^3H.
 \end{equation}
Here $E=R(\pi_1)$, $H=R(\omega_1)$, $K=R(\pi_1 + \pi_2)$ as complex modules over $\k=\sp(1)\oplus\sp(n)$,
where $\omega_1$ and $\pi_i$ are the fundamental weights of $\sp(1)$ and $\sp(n)$.
For complex simple modules $A,B$, the notation $AB$ means a real simple module which complexifies to the
complex tensor product between $A$ and $B$, for example,
$T_oM \equiv EH$.

This description of the intrinsic torsion is called the $EH$\textit{-formalism} \cite{Sal}, and the class of an
almost quaternion-Hermitian geometry is the submodule supporting the intrinsic torsion of the structure.

\subsection{Differential Geometry of the models}

In general, for $n>2$ there are 6 fundamental classes \eqref{6cl} of almost quaternion-Hermitian manifolds.
However, given an isotropy representation, not all of these possibilities can be realized by a homogeneous geometry.
In particular, for the submaximal isotropy $\h$, we have the following.

 \begin{theorem}\label{geometryclass}
Let $(M,g,Q)$ be a sub-maximally symmetric almost quaternion-Hermitian space which is not a quaternion-K\"ahler
locally symmetric space. Then it is of class $(K+E)H$, i.e.\ the structure $(g,Q)$ is quaternion-K\"ahler with torsion.
 \end{theorem}

 \begin{proof}
According to \cite{C}, for $n>2$ the full information about the first order class of the structure is given by
the exterior derivative $d\Omega$.
However, since the deRham operator is invariant, and $\Omega$ is an isotropy invariant tensor in each tangent space,
it follows that the projection of $d\Omega$ to each irreducible component of $\Lambda^5\m$ is also invariant.
The non-vanishing conditions for these components define the first order classes in $EH$-formalism.
	
Only those $\k$-components which have a trivial $\h$-submodule can admit such a non-vanishing projection;
all other projections are automatically zero. We have:
 \begin{equation}\label{geodecomp1}
\Lambda^5\m = \Lambda^5(\R\oplus \Im(\H) \oplus \H^{n-1}) =
\R\otimes\Lambda^4(\Im(\H) \oplus \H^{n-1}) \oplus \Lambda^5(\Im(\H) \oplus \H^{n-1}).
 \end{equation}
We will decompose this further into tensor products of alternating powers of $\Im(\H)$ and $\H^{n-1}$.
For $p=4,5$ we have:
 \begin{equation}\label{Lambda45}
\Lambda^p(\Im(\H)\oplus\H^{n-1}) = \sum_{k=0}^3 \Lambda^k\Im(\H)\otimes\Lambda^{p-k}\H^{n-1}.
 \end{equation}
Of these terms, only those containing a trivial representation of $\sp(n-1)$ may contribute.
These are the factors containing $\Lambda^2\H^{n-1}$ and $\Lambda^4\H^{n-1}$, with
 $$
(\Lambda^2\H^{n-1})^{\sp(n-1)} \simeq \Im(\H)\ \text{ and }\
(\Lambda^4\H^{n-1})^{\sp(n-1)} \simeq \R.
 $$
For $p=5$ the respective terms in \eqref{Lambda45} correspond to $k=1,3$ and as $\sp(1)$-modules are:
$\Lambda^1\Im(\H)\otimes\R\simeq\Im(\H)$ and $\Lambda^3\Im(\H)\otimes\Im(\H)\simeq\Im(\H)$, both non-trivial as $\h$-modules.
For $p=4$ the respective terms in \eqref{Lambda45} correspond to $k=0,2$ and as $\sp(1)$-modules are:
$\Lambda^0\Im(\H)\otimes\R\simeq\R$ and $\Lambda^2\Im(\H)\otimes\Im(\H)\supset\R$, so contain two trivial $\h$-modules.

Consequently the space of invariant 5-forms has dimension 2, and at most 2 fundamental classes in $EH$-formalism can be realized.

To see what these classes are we convert the two trivial $\h$-modules to the $EH$-formalism.
The class $EH$, when branched to $\h\subset\k=\sp(1)\oplus\sp(n)$, is isomorphic to $\m$,
and so has precisely one trivial $\h$-submodule of dimension 1.

The module $KH$ is the complex $\k$-module $R(\omega_1 + \pi_1 + \pi_2)$ and so can be equivalently written
 \begin{equation}\label{KHequ}
K = (\Lambda^2_0 E)\otimes E /( E + \Lambda^3_0 E ).
 \end{equation}
To do the branching, note that $E=R(\omega_1)\oplus R(\pi_1)$ and $H=R(\omega_1)$ as complex $\h$-modules.
Substituting the first of those into \eqref{KHequ} and computing the tensor products shows that $K$ contains exactly
one submodule $R(\omega_1)$ with respect to $\h$. Therefore $KH$ contains a trivial submodule of multiplicity 1.
This shows that the class of the geometry is a subclass of $(K+E)H$.
\end{proof}

Theorem \ref{geometryclass} simplifies investigation of the geometric class of our models in terms of the parameters. We get the following identity.
\begin{equation}
	d\Omega = f_{KH} \theta_{EH} + f_{EH} \theta_{KH}
	\label{QTdecomp}
\end{equation}
Here $\theta_{EH}$ and $\theta_{KH}$ are two linearly independent and $\h$-invariant 5-forms from the appropriate $\k$-representation (depending only on the isotropy representation, and so fixed between models), and $f_{EH},f_{KH}$ are rational functions of all parameters, meaning $c_1,c_2$, the sign $\pm$ from model $\Hh_1^\pm$, and $\beta$ from models $\Hh_3^\beta$, $\Hh_5^\beta$.
A point in the parameter space ($c_1,c_2>0$) where either $f_{EH}$ or $f_{KH}$ vanishes defines a geometry of class either $EH$ or $KH$, respectively.
A solution of the equation $f_{EH}=f_{KH}=0$ corresponds to a quaternion K\"ahler structure.
We tabulate this information, as well as the possible reductions, in Table \ref{EHtable}, where the reduction $0$ means quaternion K\"ahler, $EH$ means locally conformally quaternion K\"ahler, and $KH$ means a structure with intrinsic torsion supported in the $KH$ submodule that we call special torsion in Table \ref{diffcons}.

\begin{table}[h]
	\centering
	\begin{tabular}{l | l | l | l}
		Model & $f_{EH}$ & $f_{KH}$ & Reduction\\
		\hline
		$\Hh_1^\pm$ & $(3c_1-2c_2)(\pm c_1+2c_2)$ & $(3c_1+5c_2)(\pm c_1+2c_2)$ & $0, EH, KH$ \\
		$\Hh_2$ & $c_1(3c_1-2c_2)$ & $c_1(3c_1+5c_2)$ & $EH$ \\
		$\Hh_3^\beta$ & $c_2(\beta c_1-2\beta c_2+2c_1)$ & $c_2(\beta c_1+5\beta c_2+2c_1)$ & $EH, KH$ \\
		$\Hh_4$ & $c_2(c_1-2c_2)$ & $c_2(c_1+5c_2)$ & $EH$ \\
		$\Hh_5^\beta$ & $c_2(\beta c_1-2\beta c_2+c_1)$ & $c_2(\beta c_1+5\beta c_2+c_1)$ & $EH, KH$ \\
		$Q_{\H P^{n-1}}$ & $c_1(3c_1-4c_2)$ & $c_1(3c_1+3c_2)$ & $EH$ \\
		$Q_{\H H^{n-1}}$ & $3c_1^2$ & $c_1(3c_1+7c_2)$ & \text{no reductions}
	\end{tabular}
	\vskip9pt \caption{Class coefficients for submaximal models}
	\label{EHtable}
\end{table}

This allows to easily re-prove sub-maximality of the models.

 \begin{proof}[Alternative proof of Proposition \ref{Pr7}.]
Each maximally symmetric model has irreducible isot\-ropy representation of real type, and so
admits a single invariant metric up to homothety. This metric is quaternion K\"ahler.
Therefore, any metric which is not quaternion K\"ahler is not maximally symmetric, and by construction our cases are submaximally symmetric.

By Table \ref{EHtable}, only Model $\Hh_1^-$ equipped with metric parametrized by $c_1=2c_2$ is quaternion K\"ahler.
This Lie algebra is the parabolic subalgebra $\p$ of $\sp(1,n)$, and the intersection $\k\cap\p = \h$.
Since the $\sp(1,n)$ action on $\H H^n$ has isotropy $\k$, $\p$ acts locally transitively near a regular point with
the isotropy $\h$ and preserves the quaternion K\"ahler metric. Thus the quaternion K\"ahler metric in this exceptional
case has symmetry algebra $\sp(1,n)$, see Remark \ref{RkK} of the next subsection for more details.
 \end{proof}

\subsection{Riemannian geometric properties}

In this section we consider some purely Riemannian properties of the metrics $g_{c_1,c_2}$ of our models.
We begin with the left-invariant metrics on Lie groups given by \eqref{g}.

 \begin{lem}
Let $g$ be a Riemannian metric which is invariant with respect to a simply transitive Lie algebra
$\mathfrak{L}\simeq \R \ltimes_\eta \mathfrak{n}$, where $\mathfrak{n}$ is a non-trivial abelian subalgebra and
the action of $\R$ on $\mathfrak{n}$ is scalar by a non-zero real number $\eta$.
Then $g$ is locally equivalent to a constant sectional curvature hyperbolic metric.
 \end{lem}\label{hyperbolic}

 \begin{proof}
First note that the standard hyperbolic metric on $H^{\dim\mathfrak{n}+1}$ satisfies the assumptions
because it is invariant with respect to the radical $\mathfrak{L}$ of the parabolic algebra of the isometry
algebra $\so(1,\dim\mathfrak{n}+1)$. Thus it is a left invariant metric on the simply connected Lie group
$\exp\mathfrak{L}$.	Next, $GL(\mathfrak{n})\subset \text{Aut}(\mathfrak{L})$ acts transitively on the space
of metrics on $\mathfrak{n}$, and the complement $\R$ can be taken to be orthogonal to $\mathfrak{n}$.
Thus all such metrics are equivalent.
 \end{proof}

 \begin{prop}\label{RiPro}
Let $g$ be an invariant metric on one of the models $\Hh_1^\pm,\Hh_2,\Hh_3^\beta,\Hh_4,\Hh_5^\beta$. Then
	\begin{itemize}
\item $g$ is Einstein iff $(M,g)$ is either $\Hh_1^\pm$ for $c_1=2c_2$, or $\Hh_3^2$.
\item $g$ is conformally flat iff $(M,g)$ is either $\Hh_3^2$, or $\Hh_5^{\pm 1}$.
\item $g$ has parallel curvature, i.e.\ is a Riemannian symmetric space, iff $(M,g)$ is either $\Hh_1^{\pm}$
for $c_1=2c_2$, or $\Hh_3^0$, $\Hh_3^2$, $\Hh_4$, $\Hh_5^\beta$.
	\end{itemize}
In particular, $\Hh_1^\pm$ for $c_1=2c_2$ is homothetic to quaternionic projective space $\H H^n$.
$\Hh_3^2$ is homothetic to the Riemannian hyperbolic space $H^{4n}$ and $\Hh_3^0$ to $H^3 \times \R^{4n-3}$.
$\Hh_4$ is homothetic to $\R^3\times H^{4n-3}$.
$\Hh_5^{\beta\neq0}$ is homothetic to $S^3_k\times H^{4n-3}$ and $\Hh_5^0$ to $S^3\times \R^{4n-3}$.
 \end{prop}

In the formulation above and in the proof below we denote by $S^m_k$ the $m$-dimensional Riemannian space of constant
positive sectional curvature $k$ and by $H^m_{-k}$ the $m$-dimensional Riemannian space of constant negative sectional
curvature $-k$, both simply connected complete (for $k>0$). We also set $S^m=S^m_1$, $H^m=H^m_{-1}$.

 \begin{proof}
Consider the Lie algebra isomorphism $\sigma:\Hh_1^-\to\Hh_1^+$ obtained by multiplication of the summand $\Im(\H)$
of $\m$ by $-1$. This map does not preserve $Q$, but is an isometry of any metric \eqref{g} on $\m$.
The quaternion K\"ahler metric on $\Hh_1^-$ corresponding to $c_1=2c_2$ is Einstein and has parallel curvature,
so the same applies for its image under $\sigma$, i.e.\ the metric with $c_1=2c_2$ on $\Hh_1^+$ .
These properties do not hold for other metrics $c_1\neq 2c_2$ on $\Hh_1^\pm$.
The Lie algebra $\Hh_2$ is graded nilpotent and non-Abelian, thus it does not admit any left-invariant
Einstein metrics \cite{DM}, see also \cite[Sect.7E]{B}, neither that of parallel curvature.

For $\Hh_3^\beta$ there are several cases. First, if $\beta=2$ then $\m\simeq \mathfrak{L}$, where $\mathfrak{L}$
satisfies the conditions for Lemma \ref{hyperbolic}. Thus the space $(\Hh_3^\beta,g_{c_1,c_2})$ for any $c_1,c_2$
is isometric to $H^{4n}_{-k}$ and so any metric is Einstein and conformally flat.
If $\beta=0$ then the space is isometric to $H^3_{-k}\times\R^{4n-3}$. Finally, if $\beta\not=2,0$ then the metrics
on $\Hh_3^\beta$ do not have parallel curvature, so the model is not isometric to a Riemannian symmetric space.

Next, $\Hh_4$ has Lie algebra $\m=\R^3\oplus \mathfrak{L}$, hence the space is isometric to a Riemannian symmetric
space $\R^3\times H^{4n-3}_{-k}$. The Lie algebra of $\Hh_5^\beta$ for $\beta\neq0$ is the direct sum
$\m=\sp(1)\oplus\mathfrak{L}$, hence $(\Hh_5^\beta,g_{c_1,c_2})$ is the Riemannian product $S^3_{k'}\times H^{4n-3}_{-k''}$.
When $\beta=0$ then $\Hh_5^0$ with any metric of the family $g_{c_1,c_2}$ is $S^3_{k}\times\R^{4n-3}$.

If a metric is conformally flat, then its symmetry algebra must embed into the Lie algebra of conformal symmetries
of the flat metric, i.e.\ $\so(1,4n+1)$. Thus, the maximal solvable subalgebra of $\m$ must embed into the
radical of the parabolic subalgerba of $\so(1,4n+1)$, which has the form $\mathfrak{L}=\R\ltimes \R^{4n}$.
In particular, the maximal symmetry subalgebra of $\m$ must be either Abelian or two step solvable non-nilpotent.
This means that $\Hh_1^\pm$ and $\Hh_2$ do not admit any conformally flat metrics. The claims for the other cases are
obtained by a straightforward computation.
 \end{proof}

 \begin{rk}\label{RkK}\rm
The proposition yields the isometry $\Hh_3^2\simeq H^{4n}_{-k}$. Let us explain how to see a quaternionic structure
on $H^{4n}=SO(1,4n)/O(4n)$. Recall that it is obtained from the pseudosphere
$S_+^{0,4n}\subset\R^{1,4n}=\R(t)\times\H^n(h_1,\dots,h_n)$ given by $t^2-\sum_{k=1}^n|h_k|^2=1$, $t>0$,
by the projection to the unit ball $B^{4n}\subset\H^n$ from the point $(-1,0)\in\R\times\H^n$.
The horizon $\partial B^{4n}\simeq S^{4n-1}=(N\setminus0)/\R_+$ is the spherization of the null cone $N\subset\R^{1,4n}$,
and for the parabolic subgroup $P_H=(\R_+SO(4n-1))\ltimes\R^{4n-1}$ of $SO(1,4n)$ we have:
$\partial B^{4n}=SO(1,4n)/P_H$. The stabilizer of $P_H$ at $o\in H^{4n}$ is $St_oP_H=SO(4n-1)$.

On the other hand, $B^{4n}$ carries also the structure of the quaternion K\"ahler space of negative curvature
$\Hh^n=Sp(1,n)/Sp(1)Sp(n)$ that is obtained from the quaternion pseudosphere
$S^3\times B^{4n}\simeq S^{3,4n}\subset\R^{4,4n}=\H^{1,n}(h_0,h_1,\dots,h_n)$ given by $|h_0|^2-\sum_{k=1}^n|h_k|^2=1$
via quotient by $Sp(1)\simeq S^3$. This can be seen as the reduction $h_0\to t=\Re(h_0)$ or as a component of
the intersection $S^{3,4n}\cap\R^{1,4n}=S^{0,4n}$. The horizon can be also identified as $\partial B^{4n}=Sp(1,n)/P$,
where $P\simeq(Sp(1)\R_+Sp(n-1))\ltimes(\Im(\H)\oplus\H^{n-1})$ is the parabolic subgroup we already met.
Its stabilizer is $St_oP=Sp(1)Sp(n-1)$. Now the space of invariant quaternion K\"ahler structures on $H^{4n}$
is the principal bundle $SO(1,4n)/P\stackrel{P_H/P}\longrightarrow SO(1,4n)/P_H$ and so it consists
of a point at the horizon $\partial B^{4n}$ (choice of a parabolic $P_H$) and the structural group reduction
$SO(4n-1)/Sp(1)Sp(n-1)$ at $o$.
 \end{rk}

Finally, let us consider the proper homogeneous submaximally symmetric quaternion Hermitian spaces $Q_{\H P^{n-1}}$
and $Q_{\H H^{n-1}}$. The first of them is a bundle over the quaternion K\"ahler space $\H P^{n-1}$ with the fiber $\H_\times$.
The family $g_{c_1,c_2}$ has one parameter as a scale of the standard metric on the base and the second parameter comes
as the scale in the fiber. Topologically this $M$ is $\R_+\times S^{4n-1}$, and the first factor is flat.
Thus the metric is never Einstein. However the metric on the second factor is Einstein for precisely two values of
the relative scale $c=c_1/c_2$, see \cite{Jen}. One of them corresponds to the standard round sphere $c=1$, which is also
conformally flat and for the other this fails. Note that $(M,g)$ is Riemannian symmetric precisely for these two values of
parameters and $(M,g)$ is not conformally flat for the general value of parameters.

Similarly, $Q_{\H H^{n-1}}$ is a bundle over quaternion K\"ahler space $\H H^{n-1}$ with the fiber $\H_\times$.
Topologically it is $\R_+\times S^3\times B^{4n-4}$. Again it is never Einstein, and it has parallel curvature iff $S^3\times B^{4n-4}$
is Einstein (we have not computed this, it can be decided similar to \cite{Jen}).
This $(M,g)$ is never conformally flat. Indeed, by the argument of the
proof of Proposition \ref{RiPro} in that case the radical of the parabolic of $\sp(1,n)$ would embed into
$\mathfrak{L}=\R\ltimes \R^{4n}$. But this is impossible as this radical is 3-step solvable.

\section{Sub-Maximal Automorphism Groups}\label{S5}

If dimension of the automorphism group of $(M,g,Q)$ exceeds $d_n$ then the same is true for the symmetry algebra and so,
by the results of Section \ref{S3}, $\dim\sym(M,g,Q)=D_n$.
We will first demonstrate that this implies $\dim G=D_n$ for $G=\op{Aut}(M,g,Q)$
and then classify all almost quaternion Hermitian spaces with $\dim G\ge d_n$.

The case $n=1$ of Riemannian 4D geometry is known, so we will assume $n>1$.

\subsection{Locally maximally symmetric geometries}

We first show that large symmetry implies the absence of low-dimensional orbits, including singular orbits.

 \begin{lem}\label{LlL}
When\/ $\dim\sym(M,g,Q)\ge d_n-\delta_{2,n}=2n^2+n+4$ the space is locally homogeneous.
When\/ $\dim\op{Aut}(M,g,Q)\ge d_n-\delta_{2,n}$ the space is globally homogeneous.
 \end{lem}

 \begin{proof}
Consider at first the case\/ $\dim\sym(M,g,Q)>d_n$. Then the space is locally maximally symmetric, i.e.\ isomorphic to $\H P^n$, $\H^n$ or $\H H^n$,
near generic point. In particular, the isotropy algebra is $\sp(1)\oplus\sp(n)$. This acts irreducibly on the tangent space
$\m=T_oM$ and hence no lower-dimensional orbits except for a singular point is possible. If the orbit at at point $o\in M$
is $o$ itself then, because the action is isotropy faithful (as for any Riemannian geometry), the symmetry algebra/automorphism group embeds into
the isotropy $\sp(1)\oplus\sp(n)$, or respectively the stabilizer $Sp(1)Sp(n)$, which has dimension $2n^2+n+3<d_n$. This contradicts the assumption.

Next consider the case $\dim\sym(M,g,Q)=d_n$ for $n>2$. The possible Lie algebras were classified in Section \ref{S3} and the isotropy there is
$\h=\sp(1)\oplus\sp(n-1)$. This action is reducible, but any decrease in the tangent to the orbit $\m$ results in a reduction of $\dim\g$ below $d_n$
(because the isotropy cannot grow).
In the group case, the stabilizer can be larger, but then its representation becomes irreducible and this has been already excluded.

Finally, for $n=2$ the case of dimension $d_2=15$ gives two quaternion K\"ahler symmetric spaces, which are again isotropy irreducible.
From Section \ref{S2} we know that in the case of dimension $d_2-1=14$ the space is either quaternion K\"ahler symmetric (one more case)
or a space locally homogeneous near generic points with the isotropy $\sp(1)\sp(1)$. The same argument as above eliminates this possibility,
implying the claim.

Thus there is only open orbits of the Lie algebra or the Lie group respectively, and since the space $M$ is assumed
connected, such an orbit is unique.
Consequently, the space $M$ is locally (in the algebra case) or globally (in the group case) homogeneous.
 \end{proof}

\subsection{Classification of maximally symmetric models}

Let $G=\op{Aut}(M,g,Q)$ and $\g=\sym(M,g,Q)$ as before. If $\dim G=D_n$, then $\dim\g=D_n$ and so $M$ is locally
isomorphic to $\H P^n=Sp(n+1)/Sp(1)Sp(n)$, $\H H^n=Sp(1,n)/Sp(1)Sp(n)$ or $\H^n=Sp(1)Sp(n)\ltimes\H^n/Sp(1)Sp(n)$.
Since these spaces are simply-connected, $M$ can be only covered by one of them.
In order to preserve the dimension of the group $G$, this covering should correspond to the quotient by a central subgroup.

The center of $Sp(n+1)$ is $\pm\1$, but it belongs to the stabilizer $Sp(1)Sp(n)$, which leads to a different
representation of the homogeneous space (in fact, it is better to quotient out the center to have an effective representation).
The situation is similar with $Sp(1,n)$, while the group $Sp(1)Sp(n)\ltimes\H^n$ is center-free.

Thus we conclude that only three almost quaternion Hermitian spaces $\H P^n$, $\H^n$ and $\H H^n$ have the automorphism
group $G$ of maximal dimension $D_n$.

\subsection{Classification of sub-maximally symmetric models}

We shall classify all spaces with $\dim G\in[d_n,D_n)$. Consider at first the case $\dim\g>d_n$ implying $\dim\g=D_n$.

Thus $M$ is locally one of the three maximally symmetric spaces just classified, which we denote by $M_0$ with the corresponding group $G_0$.
It can happen that $\dim G<\dim\g=D_n$. There are two reasons for a reduction of dimension of the maximal symmetry group:
that the universal cover $\tilde M$ is incomplete or that $\tilde M$ non-trivially covers $M$ (a combination of those is also possible indeed).

In the first case, if $\tilde M\subsetneqq M_0$ then for any $x\in M$ some geodesic from $x$ is incomplete,
so $T_xM\neq G/H_x$ where $H_x$ is the stabilizer of $x$. This means that the group $G$ acts intransitively,
which by Lemma \ref{LlL} implies that $\dim G<d_n$ contradicting the assumption.
Thus $\dim G\ge d_n$ only if $\tilde M=M_0$, so the reason for a reduction of the dimension can only be
non-simply connectedness of $M$, whence the projection $p:M_0\to M$.

Since we know that $G$ acts transitively, the drop in symmetry dimension is only possible due to a
reductions of the stabilizer subgroup $H$ compatible with the isotropy representation. From Section \ref{S2}
we know that for $n>2$ this reduces the stabilizer at least to $Sp(1)Sp(n-1)$, which yields $\dim G\leq d_n$.
Consequently, no $G$ has dimension in the open interval $(d_n,D_n)$ and $d_n$ is the submaximal dimension of the
automorphism groups of almost quaternion Hermitian spaces. This justifies the symmetry dimension gap.

We can compute reduction of the stabilizer $H$ of a point $o\in M_0$ via its action on the fundamental group
$\pi_1(M,p(o))$ by monodromy, but it is easier to approach this by describing the possible large subgroups $G\subset G_0$.
For the compact maximal size group $G_0=Sp(n+1)$
its subgroup is either semi-simple or reductive and among such $Sp(1)Sp(n)$ is maximal in dimension. But for this group
$\dim<d_n$. For the group $G_0=Sp(1,n)$ the maximal subgroups, by Mostow's theorem \cite{M}, are either semi-simple or
stabilizers of pseudo-tori (reductive) or parabolic.

Among the first two the maximal in dimension is $Sp(1)Sp(n)$ eliminated above.
The parabolic subgroup $P$ is unique up to conjugation. For its action on $(M,g,Q)$ the stabilizer should be contained
in the maximal compact subalgebra $P_{ss}=Sp(1)Sp(n-1)$. The radical is solvable $P_{rad}=\R+\H^{n-1}+\Im(\H)$
with the weights of components $0,1,2$. Since $\dim P_{rad}=4n$, the only possibility is that $(g,Q)$
is the left-invariant structure on $P_{rad}$ equivariant with respect to the stabilizer $P_{ss}$. Such structures
have been classified in Section \ref{S3}, from which we know that the only $P$-invariant one is $\Hh_1^{-}$.

Finally, $G_0=Sp(1)Sp(n)\ltimes\H^n$ has several subgroups of $\dim>d_n$ (for instance, $U(1)Sp(n)\ltimes\H^n$), but
due to transitivity the $\H^n$ component of the group should persist (the action may become almost effective).
Thus the only possible $M$ are quotients of $\H^n$ by lattices, so they are tori. Each lattice
reduces the maximal stabilizer $Sp(1)Sp(n)$ to a proper subgroup
(stabilizer of the lattice). Those are easy to classify, and the maximal proper isotropy becomes $H=Sp(1)Sp(n-1)$ in the
case of 1D lattice $\Z\subset\R\subset\R\oplus\Im(\H)\oplus\H^{n-1}=\H^n$. This yields the quaternion-K\"ahler space
$S^1\times\R^{4n-1}=\H^n/\Z$ with the automorphism group $G=S^1\times Sp(1)Sp(n-1)\ltimes(\Im(\H)\oplus\H^{n-1})$
obtained as the quotient of $\R^1\times Sp(1)Sp(n-1)\ltimes(\Im(\H)\oplus\H^{n-1})\subset G_0$ by the kernel of the action.

Now we consider the last remaining case $\dim G=\dim\g=d_n$. In this case $\g$ is one of the submaximal algebras and in
the simply-connected case $(M,g,Q)$ is the left invariant quaternion Hermitian structure on the Lie group, as classified
in Section \ref{S3}. No singular orbits or incomplete domains are possible and the only quotient not reducing the dimension
of $G$ is the quotient by a central discrete subgroup.

In the case of simply-transitive structures, only those with $\beta_1=\beta_2=0$ have a center. These are $\Hh_2$ and
$\Hh_5^0$, the center in both cases is $\R_+\simeq\R$, corresponding to the $\R$ component in $\m\subset\g$. The discrete
subgroups are equivalent to $\Z\subset\R$ and we get two additional spaces $\Hh_2/\Z,\Hh_5^0/\Z$,
both diffeomorphic to $S^1\times\R^{4n-1}$, with $\dim G=d_n$.

The two proper homogeneous submaximally symmetric structures both have center isomorphic to $\R$,
namely $G$ is $\R_+ Sp(n)$ or $\R_+ Sp(1,n-1)$;
there is also $\Z_2=\{\pm\1\}$ central component in the semisimple part, but it acts trivially on $M$.
Thus we get two more spaces $Q_{\H P^{n-1}}/\Z,Q_{\H H^{n-1}}/\Z$ with the group $G$ of submaximal dimension $d_n$.
The spaces are $S^1\times S^3$-bundles over $\H P^{n-1}$ and $\H H^{n-1}\simeq B^{4n-4}$ respectively.

The special case $n=2$ gives only two quaternion K\"ahler spaces (Wolf spaces) that we already discussed.
This finishes the proof of Theorem \ref{ThB}.

 \begin{rk}\rm
From the proof of Lemma \ref{LlL} we see examples of quaternion K\"ahler symmetric spaces with the automorphism group $G$ of dimension $d_n-1$,
namely $\H P^n\setminus o$, $\H^n\setminus o$ and $\H H^n\setminus o$ for a point $o$. Thus the sub-submaximal automorphism dimension is $d_n-1$,
cf. Remark \ref{RK2}
(in the case $n=2$ the sub-submaximal dimension is attained on another quaternion K\"ahler symmetric space or on a series of constructed homogeneous spaces).
 \end{rk}


\end{document}